\documentclass{amsart}
\usepackage{amsfonts,amssymb,amscd,amsmath,enumerate,verbatim,calc}
\newtheorem{theorem}{Theorem}[section]
\newtheorem{lemma}[theorem]{Lemma}
\newtheorem{proposition}[theorem]{Proposition}
\newtheorem{corollary}[theorem]{Corollary}
\theoremstyle{definition}
\theoremstyle{definitions}
\newtheorem{definition}[theorem]{Definition}
\newtheorem{conjecture}[theorem]{Conjecture}

\theoremstyle{notations}

\newtheorem{example}[theorem]{Example}
\theoremstyle{remarks}

\newcommand{\T}{\mathrm}

\usepackage{float}
\usepackage{graphicx}
\usepackage{tikz}
\usetikzlibrary{shapes,arrows}
\newcommand{\GraphNodeDistance}{0.3 cm}
\newcommand{\GraphNodeSize}{2 pt}
\newcommand{\GraphInnerSep}{1 pt}
\newcommand{\GraphLineWidth}{0.6 pt}

\begin{document}
\author[M. Alinejad and S. fulad]
{M.~Alinejad$^{1,*}$ and S.~fulad$^{1}$}

\title[Equitable partitions for Ramanujan graphs]
{Equitable partitions for Ramanujan graphs}
\subjclass[2000]{05C50, 05C31}
\keywords{Edge signing, Good signing, Adjacency matrix, Lexicographic product}
\thanks{$^*$Corresponding author}
\thanks{E-mail addresses:, mohsenalinejad96@gmail.com, sanazfulad@gmail.com}
\maketitle

\begin{center}
{
\it $^1$Department of Mathematics, School of Science, Shiraz University, \\
P.O.Box 7146713565, Shiraz, Iran.
}
\end{center}
\vspace{0.4cm}
\begin{abstract}
For $d$-regular graph $ G ,$ 	
an edge-signing $ \sigma:E(G) 
\rightarrow
\lbrace -1,1 \rbrace$ is called a good signing if the absolute eigenvalues of adjacency matrix are at most 
$ 2 \sqrt{d-1}.$ Bilu-Linial conjectured that for each regular graph there exists a good signing. In this paper, by using new concept "Equitable Partition", we solve the conjecture \ref{L1} for some cases. We show that how to find out a good signing for special complete graphs and lexicographic product of two graphs. In particular, if there exist two good signings for graph $ G, $ then we can find a good signing for a $ 2$-lift of $ G.$
\end{abstract}
\vspace{0.5cm}

\section{Introduction}
For any graph $G$, we denote the set of all vertices and edges of $G$ by $V(G)$ and $E(G)$, respectively. 
For two vertices $u, v \in V(G),$ we denote $ u \sim v$ or $ uv $ for brevity, if $u$ and $v$ are adjacent.
The \textit{degree} of a vertex $v\in V(G)$, denoted by $\text{d}(v)$, 
is the number of adjacent vertices of $v$. 
The maximum and minimum degree of graph $ G $ is denoted by $ \Delta(G) $ and $ \delta(G),$ respectively. An \textit{edge-signing} is an 
edge-weighted graph obtained by signing $ -1$  and $ 1 $ to each edge of 
$ G .$ In this case, the sign of the edge $ uv $ is denoted by $ e(uv) .$ 
For the specific sign $ \sigma: 
E(G) 
\rightarrow
 \lbrace -1,1 \rbrace,$ we denote $ G^{\sigma} $ and $ A^{\sigma} $ for 
 the edge-signed graph and associated signed adjacency matrix, respectively. In the latter case,
  $A^{\sigma} = [a^{\sigma}_{ij} ]$ of $G^{\sigma}$ is defined as
$a^{\sigma}_{ij} = \sigma (ij),$ if $ij \in E(G)$, and $a_{ij} = 0,$ if $ij \notin
E(G).$ The \textit{spectrum} of a graph $G^{\sigma},$ denoted by $\lbrace \lambda_i \rbrace (1 \leq i \leq n, n=|V(G)| )$ is the multiset of the eigenvalues of $ A^{\sigma} .$ In addition, define 
$ \rho(G^{\sigma}) $ as the maximum value of 
$\lbrace |\lambda_i| \rbrace ,$ where 
$1 \leq i \leq n .$
A graph is called \textit{Ramanajun} if all its nontrivial eigenvalues lie in $ [-2\sqrt{d-1}, 2\sqrt{d-1}].$
Bilu and Linial proposed the the following conjecture in \cite{bilu}.
\begin{conjecture}{(Bilu-Linial)} \label{L1}
	Let $G$ be a $d$-regular graph with $d>1.$ Then $\rho(G^{\sigma})\leq 2 \sqrt {d-1},$ for some edge-signing 
	$\sigma $ of $G.$
\end{conjecture}

By omitting the regularity of the above conjecture, we have a stronger conjecture as follow.  
\begin{conjecture} \cite[Conjecture $2$]{greg}  \label{L2}
	For a graph $ G $ with $ \Delta(G)>1 ,$ there exists an edge  signing $ \sigma $ for $ G $ such that $$ \rho(G^{\sigma}) \leq 2 \sqrt {\Delta(G)-1}.$$	
\end{conjecture}\label{R2}

For the first time Lubotzky and et. al \cite{lubo} and Margulis \cite{marg} found the infinite families of Ramanajun graphs.
Adam and et. al proved the Conjecture \ref{L1} for bipartite graphs. In addition, they showed that there are infinite families of $ (c,d)$-biregular bipartite graphs such the second largest eigenvalue is less than 
$ \sqrt{c-1}+\sqrt{d-1}$ (\cite{marc}). For more information about the construction of Ramanajun graphs see \cite{burn,chiu,hyun,morg,jord,pize}.

A \textit{2-lift} is a process on a graph which create a new graph with twice vertices as follow.
For a graph $ G$ define graph 
$ G' $ with $ V(G') $ and $ E(G').$ 
For each $ u \in V(G),$ we have two vertices $ u_0$  and $ u_1$ in $ V(G').$ 
If $ uv \in V(G),$ then $ E(G')$ contains one of the following cases:
$$ \lbrace u_0v_0, u_1v_1 \rbrace $$
or
$$ \lbrace u_0v_1, u_1v_0\rbrace .$$
Hence, for a graph $ G,$ we can construct $ 2^{|E(G)|} $ new graphs by using $ 2$-lift of $ G.$
First time, Bilu-Linial used the $ 2$-lift method to construct infinite families of expanders.
Friedmen was the first one who introduced old and new eigenvalues for $ 2$-lift of graphs \cite{frie}. To achieve more information about $ 2$-lift of graphs see \cite{amit, lela, lube}.

A \textit{conference} matrix, denoted by $C(n),$ 
is a weighing matrix of order $n$ and its entries on 
diagonal are zero and off the diagonal are $ 1 $ and $ -1 $ such that 
$ C(n){C(n)}^{T} = I.$ 
 It is known that a symmetric conference matrix exists if $n=q+1,$ where $q$ is a 
prime power with $q \equiv 1$ (mod $4$). 
A symmetric conference matrix which entries on the first row are non negative is called normalized symmetric conference matrix.   
A matrix $H_{n-1}$ is defined by removing the first row 
and column of normalized symmetric conference matrix $C(n).$ For each conference matrix of order $n,$ there is 
a signing $\sigma $ such that $\rho(C(n)) \leq \sqrt {n-1}.$ 

Let $G$ and $H$ be two graphs. Then \textit{lexicographic} product of $G$ and $H$ which
is denoted by $G \circ H$ is a graph with vertex set $V (G) \times V (H)$ and two vertices
$(x, y)$ and $(z, t)$ are adjacent if and only if either $x$ is adjacent to $z$ in G or $x = z$
and $y$ is adjacent to $t$ in $H.$ A \textit{decomposition} of graph 
$ G $ is a set of edge-disjoint subgraphs of $ G$ such that each edge of $ G $ belongs to exactly one subgraph.

\section{Preliminaries} 

There exists a relation between the spectrum of $ G'$ and the spectrum of $G$ and 
$ A^{\sigma} $ as follow.
\begin{proposition} \cite[Lemma 3.1]{bilu}
	Spec($ G^{'}) =$ Spec($G$) $ \cup $ Spec($ A^{\sigma} $)
\end{proposition}
\begin{definition}
Let $G$ be a graph and $\sigma $ be an edge-signing of $G.$ Then for vertex $u$ of $G,$ define 
$$N^{+}(u)= \lbrace v | \sigma(uv)=1 \rbrace,$$

$$ N^{-}(u)= N(u) \setminus N^{+}(u) .$$
Moreover, for each subset $S \subseteq V(G),$ define 
$$\T{d}(u,S)= |N^{+}(u) \cap S|-|N^{-}(u) \cap S|.$$
\end{definition}

\begin{definition}
Let $G$ be a graph graph and $\sigma $ be an signing of it. The partitions $C_1,\cdots, C_k$ of $V(G)$ is called 
\textit{equitable partition}, if for each two integers $1 \leq i,j \leq k$ and $u \in C_i,$ the value
$\T{d}(u,C_j)$ depends only on $C_i$ and $C_j.$
\end{definition}
\begin{definition}
Given an equitable partition $ \pi$ with cells $C_1, \cdots, C_k.$ The \textit{characteristic } matrix $P$ is a 
$|V(G) |\times k$ matrix with $P_{ij}=1$ if vertex $i$ belongs to $C_j$ and $P_{ij}=0,$ otherwise.
Given an equitable partition $\pi .$ Then define matrix $B$ such that $B_{ij}=\T{d}(u, C_j),$ where $u \in C_i.$
\end{definition}
Notice that the entries of $B$ can be negative. The following proposition is similar to
the \cite[Lemma 9.3.1]{godsil}. Thus, we skip the proof.

\begin{proposition} \label{a1}
Let $\pi$ be an equitable partition for $G^{\sigma}$. Then $A^{\sigma}P=PB.$
\end{proposition}

By Proposition \ref{a1}, we have ${(A^{\sigma})}^{r}P=PB^{r},$ where $ r$ is a positive number. Since the entries of $B$ can be negative, it is possible the largest eigenvalue of $B$ and $A^{\sigma}$ be different.
If $X$ is an eigenvector corresponded to eigenvalue $\lambda $ of matrix $A^{sigma},$ then $A^{\sigma}X=\lambda X.$ 
By proposition \ref{a1}, we deduce that $B^{T}(P^TX)= \lambda (P^TX).$ Thus, $\lambda$ is an eigenvalue of $B$ if and only if $P^TX \neq 0.$ To find the spectrum $ G^{\sigma} ,$ we first characterize the spectrum of $ B .$ After that compute $ P^{T}X,$ where $ X $ is eigenvector correspond to $ \lambda .$ If $ P^{T}X $ is nonzero, then $ \lambda $ is eigenvalue of $ A^{\sigma} .$ If $P^{T}X $ is zero, we try to find a relation between entries of $ X $ and 
the eigenvalue $ \lambda .$

\section{Good signing for some complete graphs}
By an equitable partition 
we show that if there is a conference matrix of order 
$n,$ 
then there exists a good signing for complete graph $K_m,$ where $ m \in \lbrace n+1, n+2,n+3 \rbrace.$
\begin{lemma} \label{R1}
Let $C(n)$ be a normalize symmetric conference matrix. Then there is a good signing for $K_m,$ where 
$ m \in \lbrace n+1, n+2,n+3 \rbrace.$
\end{lemma}
\begin{proof}
Since we have three choices for $ m ,$ we consider three cases.\\
\textbf{Case 1.} Suppose that $m=n+1.$ Define an equitable partition with three cells such that $|C_1|=1, |C_2|=1$ and 
$|C_3|=n-1.$ Moreover, 
$A_{C_3}=H_{n-1}$ and for each $u \in C_3,$ define $d(u,C_i)=1,$ where $i= 1,2.$ Moreover, the sign of the edge between $ C_1 $ and $ C_2 $ is one. In this sense, 
$$
B= 
\left[ {\begin{array}{*{20}{c}}
{0}&{1}&{n-1}\\
{1}&{0}&{n-1}\\
{1}&{1}&{0}
\end{array}} \right]
$$
Suppose that $\lambda $ is an eigenvalue of $A^{\sigma}$ and $A^{\sigma}X= \lambda X.$ If $P^TX \neq \textbf{0},$ 
then $\lambda$ is an eigenvalue of $B.$ Since the eigenvalues of 
$B$ are $\frac{1}{2} (1- \sqrt{ 8n-15}), -1, \frac{1}{2} (\sqrt{ 8n-15}+1)
,$ we can see $\lambda \leq 2 \sqrt {n-1}.$ If $P^TX = \textbf{0},$ then $x_1=x_2= 0$ and 
$\lambda $ is an eigenvalue of $A_{C_3}.$ From the fact that the largest eigenvalue of $A_{C_3}$ is 
$\sqrt{n-1},$ we get $ \sigma $ is a good signing for $ K_{n+1}.$\\ 
\textbf{Case 2.} 
Assume that $m=n+2.$ Consider an equitable partition by four cells such that 
$|C_1|=|C_2|=|C_3|=1$ and $|C_4|=n-1.$ Put $A_{C_4}= H_{n-1}$ and $d(u,C_i)=1,$ for each 
$u \in C_4$ and $1 \leq i \leq 3.$ In addition, assume that the sign of the edges between $ C_1 ,$ $ C_2 $ and $ C_3 $ is $ 1 .$  Hence, 
$$
B= 
\left[ {\begin{array}{*{20}{c}}
{0}&{1}&{1}&{n-1}\\
{1}&{0}&{1}&{n-1}\\
{1}&{1}&{0}&{n-1} \\
{1}&{1}&{1}&{0}
\end{array}} \right]
$$
Assume that 
$\lambda $ is an eigenvalue of $A^{\sigma}$ and $A^{\sigma}X= \lambda X.$ If $P^TX \neq \textbf{0},$ then $ \lambda $ is an eigenvalue of $ B $. Since the eigenvalues of $B$ are $- \sqrt{3n-2}+1, -1, -1, \sqrt{3n-2}+1,$ 
we deduce that $ \lambda $ is less than $2 \sqrt{n}.$ Otherwise, we have 
$ x_1=x_2=x_3=0 $ and $ \lambda $ is an eigenvalue of 
$ H_{n-1} $ which is less than 
$2 \sqrt{n}.$ \\
\textbf{Case 3.}
Assume that
$m= n+3.$ Consider an equitable partition with three cells such that $ 
C_1 = \lbrace u_1, v_1\rbrace, C_2=\lbrace u_2,v_2 \rbrace$ and $|C_3|=n-1.$ 
Now, define a signing $\sigma$ on $E(K_m)$ such that $A_{C_{4}}= H_{n-1},$ 
$\sigma(u_1v_1 )= -\sigma(u_2v_2 )=1, \sigma(u_1u_2 )= 
\sigma(v_1v_2)=  -\sigma(u_1v_2 )= 
-\sigma(v_1u_2)=1$ and $d(u, C_i)=2,$ for each 
$u \in C_3$ and $i=1,2.$ In this case,  
 $$ B= 
\left[ {\begin{array}{*{20}{c}}
{-1}&{0}&{n-1}\\
{0}&{1}&{n-1}\\
{2}&{2}&{0}
\end{array}} \right]
$$
Let $\lambda $ be an eigenvalue of $K^{\sigma}_m$ with $A^{\sigma}X= \lambda X.$ If $P^T X \neq 0,$ then $\lambda$ 
is an eigenvalue of $B.$ Since the eigenvalues of $B$ are $- \sqrt{4n-3}, 0, \sqrt{4n-3},$ we have 
$|\lambda| \leq 2 \sqrt{n+1}.$ If $P^T X = 0,$ then $2x_2=(\lambda -1)x_1$ and $2x_1= (\lambda-1)x_2.$
If $x_1=0,$ then $\lambda $ is an eigenvalue of $A_{C_4}.$ If 
$x_1 \neq 0,$ then $\lambda =-1$ or $\lambda=3.$ Therefore, the assertion holds.
\end{proof}
\begin{example}
Consider the normalized symmetric conference matrix of order $ 6 $ given by 
$$
\left[ {\begin{array}{*{20}{c}}
	0&1&1&1&1&1\\
	1&0&1&-1&-1&1\\
	1&1&0&1&-1&-1\\
	1&-1&1&0&1&-1\\
	1&-1&-1&1&0&1\\
	1&1&-1&-1&1&0\\
\end{array}} \right]
$$
By Lemma \ref{R1}, we have good signings $ \sigma_{1}, \sigma_{2} $ and $ \sigma_3 $ for $ K_7 , K_8$ and $ K_9 ,$ respectively, such that their adjacency matrices are as below. 
$$
A^{\sigma_{1}}=
\left[ {\begin{array}{*{20}{c}}
		0&1&1&1&1&1&1\\
		1&0&1&1&1&1&1\\
		1&1&0&1&-1&-1&1\\
		1&1&1&0&1&-1&-1\\
		1&1&-1&1&0&1&-1\\
		1&1&-1&-1&1&0&1\\
		1&1&1&-1&-1&1&0\\
\end{array}} \right]
$$
,
$$
A^{\sigma_{2}}=
\left[ {\begin{array}{*{20}{c}}
		0&1&1&1&1&1&1&1\\
		1&0&1&1&1&1&1&1\\
		1&1&0&1&1&1&1&1\\
		1&1&1&0&1&-1&-1&1\\
		1&1&1&1&0&1&-1&-1\\
		1&1&1&-1&1&0&1&-1\\
		1&1&1&-1&-1&1&0&1\\
		1&1&1&1&-1&-1&1&0\\
\end{array}} \right]
$$
,
$$
A^{\sigma_{3}}=
\left[ {\begin{array}{*{20}{c}}
		0&-1&1&-1&1&1&1&1&1\\
		-1&0&-1&1&1&1&1&1&1\\
		1&-1&0&1&1&1&1&1&1\\
		-1&1&1&0&1&1&1&1&1\\
		1&1&1&0&1&1&1&1&1\\
		1&1&1&1&0&1&-1&-1&1\\
		1&1&1&1&1&0&1&-1&-1\\
		1&1&1&1&-1&1&0&1&-1\\
		1&1&1&1&-1&-1&1&0&1\\
		1&1&1&1&1&-1&-1&1&0\\
\end{array}} \right]
$$
\end{example}
\section{Good signing for lexicographic product of graphs}
In this section, 
we consider graphs $G \circ \overline {K_2}$ 
and $G \circ \overline {K_4}$ to find relation between edge-signing 
of $G$ and its lexicographic product. In the former case, we show that for a non-bipartite regular graph $ G ,$ if $ G $ is decomposable to two regular bipartite graphs, then there exists a good signing for $ G \circ \overline {K_2}.$ In the latter case, if $ G $ is regular with a good signing, then $ G \circ \overline {K_4}$ has a good signing.
First consider $ G \circ \overline {K_2}.$ Assume that $V(\overline {K_2})= 
\lbrace u,v \rbrace .$ For each $x,y \in V(G)$ that $x$ is adjacent to $y,$ 
we consider three edge-signing families
for the induced subgraph on vertices $\lbrace x,y \rbrace  \times \lbrace u,v \rbrace,$ which 
is isomorphic to $C_4,$ as bellow:
\begin{itemize} \item[(i)] 
All edges have signing $1$.
\item[(ii)] All edges have signing $-1$.
\item[(iii)] sign the edges alternatives by $1,-1$
\end{itemize} 
Now, consider an arbitrary signing $(G \circ \overline {K_2})^{\sigma}$ such 
that $\sigma$ satisfies in above condition and 
define each cell as $\lbrace x \rbrace  \times \lbrace u,v \rbrace ,$ where $x \in V(G).$ Therefore, 
$B$ is a matrix of order $n$ and $B_{ij} \in \lbrace-2,0,2 \rbrace.$ Now, define subgraphs 
$H_1$ and $H_2$ of $G$ such that $V(H_1)=V(G)$ and $xy \in E(H_1)$ if $xy \in E(G)$ and 
edge-signing of the induced subgraph 
$\lbrace x,y \rbrace  \times \lbrace u,v \rbrace$ is not satisfy in $(iii),$ and 
$V(H_2)=V(G)$ and $E(H_2)=E(G) \setminus E(H_1).$
Now, define signings $H_1^{\sigma_1}$ and  $H_2^{\sigma_2}$ from $\sigma$ as bellow:
\begin{itemize}
\item[(a)] $\sigma_1(xy)=\sigma(xu),$ if $xy \in E(H_1).$

\item[(b)] $\sigma_2(xy)=\sigma(xu),$ if $xy \in E(H_2).$
\end{itemize}
It is not difficult to see that $ A^{\sigma_{1}}=\frac{1}{2} B$.
Suppose that 
$A^{\sigma}X= \lambda X,$  
Depending on $ P^TX $ is zero or not, consider the following two cases:\\
\textbf{Case 1.}
If $P^TX \neq \textbf{0},$ then $\lambda$ is an eigenvalue of $B.$ Hence, $2 \lambda$ is an eigenvalue of 
$H_1^{\sigma_1}.$ \\
\textbf{Case 2.}
If $P^TX = \textbf{0},$ then we get $x_i=-x_{i+1},$ for each 
odd number $i, 1 \leq i \leq 2n.$ By above signing, for the adjacent vertices $x$ and $y$ in $V(G),$ the submatrix 
correspond to the vertices $\lbrace x,y \rbrace  \times \lbrace u,v \rbrace $ is one of the following matrices: 
$$\left[ {\begin{array}{*{20}{c}}
1&1\\
1&1
\end{array}} \right]
\left[ {\begin{array}{*{20}{c}}
-1&-1\\
-1&-1
\end{array}} \right]
\left[ {\begin{array}{*{20}{c}}
1&-1\\
-1&1
\end{array}} \right]
\left[ {\begin{array}{*{20}{c}}
-1&1\\
1&-1
\end{array}} \right]
$$
By multiplying the fist or second matrices by the 
vector $ [x_i,x_{i+1}]^{T} ,$ we get 
the vector $ \textbf{0} .$ If we multiply the third or fourth matrices by 
$ [x_i,x_{i+1}]^{T} ,$ then we have  $ 2[x_i,x_{i+1}]^{T}$ or 
$ -2[x_i,x_{i+1}]^{T},$ respectively. 
Now, by multiplying $A^{\sigma}$ 
by $X,$ we deduce that $\lambda$ is twice of an eigenvalue of matrix 
corresponded to graph $H_2^{\sigma_2}.$ Thus, we can deduce that 
$\rho(G^{\sigma}) \leq 2 \ \T{max} \lbrace \rho(H_1^{\sigma_{1}}), \rho(H_2^{\sigma_{2}})\rbrace.$

Assume that $G$ is a non-bipartite graph and it decomposes to two bipartite graphs $H_1$ and $H_2.$
\begin{lemma}
Let non-bipartite graph $G$ be decomposable to $H_1$ and $H_2$ such that 
$\Delta(G)=2\Delta(H_1)=2\Delta(H_2).$ If there are good signings for $H_1$ and $H_2,$ then 
there is a good signing for $G \circ \overline {K_2}.$ 
\end{lemma}
\begin{corollary}
If $G$ is a non-bipartite $d$-regular graph and decomposable to $\frac{d}{2}$-regular bipartite graphs $H_1$ 
and $H_2,$ then $G \circ \overline {K_2}$ has a good signing. 
\end{corollary}

\begin{example}
Consider two cycles of length six $H_1$ and $H_2$ 
such that they isomorphic to
$ u_1v_1u_2v_2u_3v_3u_1$ and $ u_1v_2v_1v_3u_2u_3u_1,$ respectively $ ($Fig. $1)$. Clearly, $ H_1 $ and $ H_2 $ are 
$ 2 $-regular bipartite graphs of a non-bipartite $ 4 $-regular graph Fig. $ 1.$
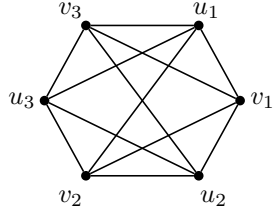
\begin{figure}[H]
	\centering
	\begin{tikzpicture}[node distance=\GraphNodeDistance, >=stealth',
		minimum size=\GraphNodeSize, inner sep=\GraphInnerSep, line width=\GraphLineWidth]
		\tikzstyle{init} = [pin edge={to-, thin, white}]
		\tikzstyle{place}=[circle, draw ,thick,fill=black]
		\tikzstyle{label}=[circle , minimum size=1 pt,thick]
		
		\node [place] (v1) at (0.75,0) {};
		\node [label] at (0.85,0.2) {$u_1$};
		
		\node [place] (v2) at (-0.75,0) {};
		\node [label] at (-0.95,0.2) {$v_3$};

		\node [place] (v3) at (1.3,-1) {};
		\node [label] at (1.6,-1) {$v_1$};

		\node [place] (v4) at (-1.3,-1) {};
		\node [label] at (-1.6,-1)
        {$u_3$};

		\node [place] (v5) at (.75,-2) {};
		\node [label] at (0.95,-2.3) {$u_2$};

		\node [place] (v6) at (-0.75,-2) {};
		\node [label] at (-0.95,-2.3) {$v_2$};

		\draw[-] (v1) -- (v2);
	    \draw[-] (v2) -- (v4);
	    \draw[-] (v6) -- (v4);
	    \draw[-] (v6) -- (v5);
	    \draw[-] (v5) -- (v3);
	    \draw[-] (v1) -- (v3);

	    \draw[-] (v1) -- (v6);
	    \draw[-] (v6) -- (v3);
	    \draw[-] (v3) -- (v2);
	    \draw[-] (v5) -- (v2);
	    \draw[-] (v5) -- (v4);
	    \draw[-] (v4) -- (v1);

		
		
		\%node [label] (uk) at (2,-0.5) [right] {};

		

		
	\end{tikzpicture}
	\caption{A non-bipartite graph of order four }
\end{figure}
\end{example}
\begin{theorem}
If there is a good signing for a $d$-regular graph $G,$ 
then there is a good signing for $G \circ \mathop {K_4}\limits^{-}.$
\end{theorem}
\begin{proof}
Suppose that $V(\overline{K_4})= 
\lbrace u_1,u_2, u_3,u_4 \rbrace $ and $\sigma $ is a signing of $G$ such that $\rho(G^{\sigma}) \leq 2 \sqrt {d-1}.$ If $xy \in E(G),$ then 
the subgraph induced by $\lbrace x,y \rbrace  \times \lbrace  u_1,u_2, u_3,u_4 \rbrace$ is 
isomorphic to $K_{4,4}.$ Now, let $xy \in E(G)$ and define a signing
$\sigma'$ of $G \circ \overline {K_4}$ as follow. If $\sigma(xy)=1$ or $\sigma(xy)=-1,$ then 
sign of the subgraph induced by $\lbrace x,y \rbrace  \times \lbrace  u_1,u_2, u_3,u_4 \rbrace$ 
is defined as $(i)$ and $(ii),$ respectively.
\begin{itemize} \item[(i)] 
All edges have signing $1$ except the edges between $(x,u_i)$ and $(y,u_i),$ for $1 \leq i \leq 4.$
\item[(ii)] All edges have signing $-1$ unless the edges between $(x,u_i)$ and $(y,u_i),$ for $1 \leq i \leq 4.$
\end{itemize} 
For each $x \in V(G)$ assume that $\lbrace x \rbrace  \times \lbrace u_1,u_2, u_3,u_4 \rbrace $ 
is one cell. By the signing $\sigma',$ this partition is an equitable partition. In this case, $B$ is a matrix of order 
$n$ and $B_{ij} \in \lbrace -2,0,2 \rbrace .$ Now, let $\lambda $ be an eigenvalue of $(G \circ \overline {K_4})^{\sigma'}.$ 
Then there exists non-zero vector $X$ such that $A^{\sigma}X= \lambda X.$
Now, consider the following two cases:\\
\textbf{Case 1.}
If $P^TX \neq \textbf{0},$ then $\lambda$ is an eigenvalue of $B.$
Since $ B=2A^{\sigma},$ we have $\lambda =2 \lambda', $ 
where $\lambda'$ is an eigenvalue of 
$G^{\sigma}.$ \\
\textbf{Case 2.}
If $P^TX = \textbf{0},$ then 
by using corresponded entries of $ X $
for each cell we have
$x_i+x_{i+1}+x_{i+2}+x_{i+3}=0,$ where $i =4q+1$ and $ 0 \leq q \leq n-1.$ 
Clearly, the submatrix 
correspond to the vertices $\lbrace x,y \rbrace  \times \lbrace u_1,u_2, u_3,u_4 \rbrace $ is one of the following matrices: 
$$\left[ {\begin{array}{*{20}{c}}
		
-1&1&1&1\\
1&-1&1&1\\
1&1&-1&1\\
1&1&1&-1
\end{array}} \right] 
\left[ {\begin{array}{*{20}{c}}
1&-1&-1&-1\\
-1&1&-1&-1\\
-1&-1&1&-1\\
-1&-1&-1&1
\end{array}} \right]
$$
In this case, $\lambda = -2 \lambda',$ where $\lambda'$ is an eigenvalue of $G^{\sigma}.$ Thus, we can state 
$\rho(A^{\sigma'})=2 \rho(A^{\sigma}).$ From the fact $ \rho(A^{\sigma}) \leq 2 \sqrt {d-1},$ we deduce that 
$$\rho(A^{\sigma'})=2\rho(A^{\sigma}) \leq 2 \sqrt {4d-1}.$$
\end{proof}
\begin{definition}
	Suppose that $ \sigma $ and $ \sigma' $ be two good signing for graph $ G.$ If there is a diagonal matrix $ D $ with entries of $ 1 $ and $ -1 $ such that $DA^{\sigma}D=A^{\sigma'},$ then $ \sigma $ and $ \sigma' $ is \textit{equivalent}. 
\end{definition}
\begin{definition}
	Let $ A $ and $ B $ be two matrices of order $ n.$ Then define a new product $ A \ast B $ as follow:
	$$ [A\ast B]_{ij}=[A]_{ij}[B]_{ij} $$
\end{definition}
Let $\sigma$ and $\sigma'$ be two good signings for $G$ which are not equivalent.
Now, we define a new signing $\tau$ for $G$ such that 
adjacency matrix of $G^{\tau}$ is equal to 
$ A^{\sigma} \ast A^{\sigma'}.$ In other words, the sign of $ uv \in E(G) $ equals
$ \sigma(uv) \sigma'(uv).$
Consider a $2$-lift graph of 
$G$ by $\tau,$ say $G'.$ Assume that 
$$V(G')= \bigcup\limits_{v \in V(G)} {\lbrace v_0, v_1\rbrace} 
.$$ Now, define a new signing  $\phi$ for graph $G'$ such that the sign of the subgraph induced by $ \lbrace u_0,v_0,u_1,v_1\rbrace $ is the same as $ \sigma
'(uv) $
 
\begin{theorem} \label{a2}
If $\sigma$ and $\sigma'$ are two good signings for $G,$ then $\phi$ is a good signing for $G'.$
\end{theorem}
\begin{proof}
For each $ v \in V(G),$ suppose that ${\lbrace v_0, v_1\rbrace}$ is a cell of $G'$. In this case, $ G' $ has $ n $ cells as $ C_1, \cdots ,C_n .$ This partition is equitable and for each vertex $ v_i \in V(G'),$ we have 
$ d(v_i, C_j) \in \lbrace 1,-1 \rbrace .$ Clearly, $ B $ is a matrix of order $ n $ and is equal to adjacency matrix of $ G^{\sigma'}.$
Suppose that $\lambda$ is an eigenvalue of $G'^{\phi}$ and $X$ is a non-zero vector of $\lambda.$
Depends on $ P^TX$ is zero or not, we have the following cases:\\
\textbf{Case 1.}
 Assume that $P^TX$ is non-zero. Hence, $\lambda$ is an 
eigenvalue of matrix $B.$ Since 
$B= A^{\sigma'},$ we deduce that 
$\lambda$ is an eigenvalue of 
$ A^{\sigma'}.$ By the fact that, $ \sigma' $ is a good signing for $ G $, we get 
$ |\lambda|< 2 \sqrt{d-1} .$
\\
\textbf{Case 2.} 
Suppose that $P^TX= \textbf{0}.$ 
Since 
$$
P^T=
\left[ {\begin{array}{*{20}{c}}
		
		1&1&0&\ldots&0\\
		0&1&1&\ldots&0\\
		&&\vdots \\
		0& \ldots &0&1&1
\end{array}} \right]_{n \times 2n}
$$ 
and
$$X^T=[x_1,x_2, \cdots, x_{2n-1},x_{2n}],$$
we have 
$x_{i}=-x_{i+1},$ for each 
odd number $i, 1 \leq i \leq 2n.$ 
So, we have 
$$ X^T=[x_1,-x_1,x_3,-x_3, \cdots, x_{2n-1},-x_{2n-1}] .$$
By removing even rows of $ A^{\phi},$ we have a matrix of order $ n \times 2n,$ say $ D .$ It is not difficult to see that 
$$ DX= (A^{\tau} \ast A^{\sigma'}) [x_1,x_3, \cdots, x_{2n-1}]^T .$$ On the other hand, 
$$ (A^{\tau} \ast A^{\sigma'})= (A^{\sigma} \ast A^{\sigma'}) \ast A^{\sigma'}= A^{\sigma} \ast (A^{\sigma'} \ast A^{\sigma'}) .$$ 
Since $A^{\sigma'} \ast A^{\sigma'}=A,$ where $ A $ is adjacency matrix of non-signing $ G.$ The entries of matrix $ A $ is one or zero. Hence, 
$$ (A^{\tau} \ast A^{\sigma'})= A^{\sigma}.$$ 
Since eigenvalues of $ A^{\phi} $ is the same as $ D ,$ and 
$$ DX= (A^{\sigma} ) [x_1,x_3, \cdots, x_{2n-1}]^T,$$ 
we can infer that $\lambda$ is an eigenvalue of 
matrix $A^{\sigma}.$
From the assumptions, $\rho(A^{\sigma}), \rho(A^{\sigma'}) \leq 2 \sqrt {d-1}$ and we 
deduce that $ \rho(A^{\tau}) \leq 2 \sqrt {d-1}.$
\end{proof}
Let $G''$ be a $2$-lift graph that is obtained from $G'$ by signing $\phi.$ Now, we can 
state the following theorem.
\begin{example}
	Consider graph $ G $ with four vertices as Fig. .....

	\begin{figure}[H]
		\centering
		\begin{tikzpicture}[node distance=\GraphNodeDistance, >=stealth',
			minimum size=\GraphNodeSize, inner sep=\GraphInnerSep, line width=\GraphLineWidth]
			\tikzstyle{init} = [pin edge={to-, thin, white}]
			\tikzstyle{place}=[circle, draw ,thick,fill=black]
			\tikzstyle{label}=[circle , minimum size=1 pt,thick]
			
			\node [place] (v1) at (0,0) {};
			\node [label] at (0,0.3) {$u$};
			
			\node [place] (v2) at (2,0) {};
			\node [label] at (2,0.3) {$v$};

			\node [place] (v3) at (2,-2) {};
			\node [label] at (2,-2.3) {$w$};

			\node [place] (v4) at (0,-2) {};
			\node [label] at (0,-2.3)
			{$z$};

			\draw[-] (v1) -- (v2);
			\draw[-] (v2) -- (v3);
			\draw[-] (v3) -- (v4);
			\draw[-] (v1) -- (v4);
			\draw[-] (v2) -- (v4);

			\%node [label] (uk) at (2,-0.5) [right] {};

		\end{tikzpicture}
		\caption{Graph $ G .$ }
	\end{figure}
\end{example}
Suppose that $ \sigma $ and $ \sigma'$ are two good edge-signing such that 
$$
A^{\sigma}=
\left[ {\begin{array}{*{20}{c}}
		
		0&1&0&1\\
		1&0&1&-1\\
		0&1&0&1 \\
		1&-1&1&0
\end{array}} \right], \  \ 
A^{\sigma'}=
\left[ {\begin{array}{*{20}{c}}
		
		0&-1&0&1\\
		-1&0&1&-1\\
		0&1&0&1 \\
		1&-1&1&0
\end{array}} \right]
$$
By Theorem \ref{a2}, we have 
$$
A^{\tau}= A^{\sigma} A^{\sigma'}=
\left[ {\begin{array}{*{20}{c}}
		
		0&-1&0&1\\
		-1&0&1&1\\
		0&1&0&1 \\
		1&1&1&0
\end{array}} \right]
$$
In this case, graph $ G'$ as Fig.....
\begin{figure}[H]
	\centering
	\begin{tikzpicture}[node distance=\GraphNodeDistance, >=stealth',
		minimum size=\GraphNodeSize, inner sep=\GraphInnerSep, line width=\GraphLineWidth]
		\tikzstyle{init} = [pin edge={to-, thin, white}]
		\tikzstyle{place}=[circle, draw ,thick,fill=black]
		\tikzstyle{label}=[circle , minimum size=1 pt,thick]
		
		\node [place] (v1) at (0,1) {};
		\node [label] at (0,1.3) {$u_0$};
		
		\node [place] (v2) at (2,0) {};
		\node [label] at (2,0.3) {$v_0$};

		\node [place] (v3) at (2,-2) {};
		\node [label] at (2,-2.3) {$w_0$};

		\node [place] (v4) at (0,-2) {};
		\node [label] at (0,-2.3)
		{$z_0$};
			
		\draw[-] (v2) -- (v3);
		\draw[-] (v3) -- (v4);
		\draw[-] (v1) -- (v4);
		\draw[-] (v2) -- (v4);

			\node [place] (w1) at (4,0) {};
		\node [label] at (4,0.3) {$u_1$};
		
		\node [place] (w2) at (6,1) {};
		\node [label] at (6,1.3) {$v_1$};

		\node [place] (w3) at (6,-2) {};
		\node [label] at (6,-2.3) {$w_1$};

		\node [place] (w4) at (4,-2) {};
		\node [label] at (4,-2.3)
		{$z_1$};
		
		\draw[-] (w2) -- (w3);
		\draw[-] (w3) -- (w4);
		\draw[-] (w1) -- (w4);
		\draw[-] (w2) -- (w4);
		
		
		\draw[-] (w2) -- (v1);
		\draw[-] (w1) -- (v2);

		\%node [label] (uk) at (2,-0.5) [right] {};
		
	\end{tikzpicture}
	\caption{Graph $ G .$ }
\end{figure}
In this sense, we have 
$$
A^{\phi}=
\left[ {\begin{array}{*{20}{c}}
		
		0&0&0&-1&0&0&1&0\\
		0&0&-1&0&0&0&0&1\\
		0&-1&0&0&1&0&-1&0 \\
		-1&0&0&0&0&1&0&-1 \\
		0&0&1&0&0&0&1&0 \\
		0&0&0&1&0&0&0&1 \\
		1&0&-1&0&1&0&0&0 \\
		0&1&0&-1&0&1&0&0 \\
		
\end{array}} \right]
$$
Note that the eigenvalues of $ A^{\phi} $ are $ \lbrace \frac{\sqrt{-17}-1}{2}, -2, -1, -1, 0, 1, 2, \frac{\sqrt{17}-1}{2}  \rbrace .$ Therefore, 
$$ \rho(G'^{\phi})= \frac{\sqrt{-17}-1}{2}< 2\sqrt{3-1}.$$ 
We can see $ \phi $ is a good signing for $ G'.$
			
\end{document}